\documentclass[11pt,reqno]{amsart}
\usepackage{amsmath,amsthm,amssymb,mathrsfs,stmaryrd,color,fullpage}
\usepackage[all]{xy}
\usepackage{url}

\usepackage{geometry}
\usepackage{enumitem}
\setlist[enumerate]{itemsep=2pt,parsep=2pt,before={\parskip=2pt}}

\usepackage{relsize}
\usepackage[bbgreekl]{mathbbol}
\usepackage{amsfonts}
\DeclareSymbolFontAlphabet{\mathbb}{AMSb}
\DeclareSymbolFontAlphabet{\mathbbl}{bbold}

\usepackage[colorlinks=true,hyperindex, linkcolor=magenta, pagebackref=false, citecolor=cyan]{hyperref}

\newcommand{\cosimp}[3]{\xymatrix@1{#1 \ar@<.4ex>[r] \ar@<-.4ex>[r] & {\ }#2 \ar@<0.8ex>[r] \ar[r] \ar@<-.8ex>[r] & {\ } #3 \ar@<1.2ex>[r] \ar@<.4ex>[r] \ar@<-.4ex>[r] \ar@<-1.2ex>[r] & \cdots }}

\newcommand{\colim}{\mathop{\mathrm{colim}}}

\newcommand{\adjunction}[4]{\xymatrix@1{#1{\ } \ar@<0.3ex>[r]^{ {\scriptstyle #2}} & {\ } #3 \ar@<0.3ex>[l]^{ {\scriptstyle #4}}}}

\newtheorem{theorem}{Theorem}[section]
\newtheorem{proposition}[theorem]{Proposition}
\newtheorem{lemma}[theorem]{Lemma}
\newtheorem{corollary}[theorem]{Corollary}

\theoremstyle{definition}
\newtheorem{definition}[theorem]{Definition}

\newtheorem{remark}[theorem]{Remark}

\newtheorem{example}[theorem]{Example}

\newtheorem{notation}[theorem]{Notation}

\newtheorem{construction}[theorem]{Construction}

\newtheorem*{convention*}{Conventions}

\begin{document}
\title{Vanishing of cohomology in  infinitely ramified towers}
\author{Bhargav Bhatt}
\begin{abstract}
We give a new proof of vanishing result of Esnault for the cohomology of constructible sheaves in the tower of ``mock'' Frobenius covers of projective space. The key idea is to use (a global form of) the perversity of nearby cycles.
\end{abstract}

\maketitle
\setcounter{tocdepth}{1}
\tableofcontents

We fix the following notation throughout the note.

\begin{notation}
Let $k$ be an algebraically closed field. Fix a prime $p$ invertible on $k$ and a finite coefficient ring $\Lambda$ of $p$-power order. For any qcqs scheme $X$, write $D(X) = D(X_{et},\Lambda)$ for the usual (unbounded) derived category of sheaves of $\Lambda$-modules on $X_{et}$, and let $D^b_{cons}(X) \subset D(X)$ be the full subcategory of constructible complexes with perfect stalks. All cohomology is \'etale cohomology.
\end{notation}

Given a proper variety $X/k$, the \'etale cohomological dimension of constructible sheaves of $\Lambda$-modules is $2\dim(X)$. On the other hand, if $X$ is affine, then the cohomological dimension is bounded by $\dim(X)$ by the Artin vanishing theorem. Starting with work of Scholze (see Remark~\ref{rmkhistory} for more), it was recently observed that certain towers of finite covers of $p$-power degree of projective varieties behave, from the perspective of cohomological dimension, as though they were affine in the limit. In this note, we focus on one instance of this phenomenon. Specifically, our goal is to offer another proof of the following vanishing theorem:

\begin{theorem}[Scholze, Esnault, Reinecke]
\label{PnVanishingIntro}
Let $\phi:\mathbf{P}^n \to \mathbf{P}^n$ be the standard toric Frobenius lift, i.e., $\phi([x_0,...,x_n]) = [x_0^p,...,x_n^p]$. For any constructible \'etale sheaf $F$ of $\Lambda$-modules on $\mathbf{P}^n$, we have
\[ \colim_i  H^j(\mathbf{P}^n, (\phi^i)^* F) = 0\]
for $j > \dim \mathrm{Supp}(F)$. 
\end{theorem}

A brief history of Theorem~\ref{PnVanishingIntro} is discussed in the forthcoming Remark~\ref{rmkhistory}. Applying this theorem to the constant sheaf on a subvariety leads to the following key example: 

\begin{example}[The case of subvarieties]
\label{exsubvar}
Let $Z \subset \mathbf{P}^n$ is a closed subvariety of dimension $d$. For each $m \geq 0$, write $Z_m = (\phi^m)^{-1}(Z) \subset \mathbf{P}^n$ for the preimage.  Then Theorem~\ref{PnVanishingIntro} implies that 
\[ \colim_m H^{> d}(Z_m,\Lambda) = 0.\]
In particular, for each $\alpha \in H^{> d}(Z,\Lambda)$, there is $m \gg 0$ such that $(\phi^m)^* \alpha = 0$ in $H^{> d)}(Z_{m},\Lambda)$.  As the Chern class $c_1(\mathcal{O}(1)) \in H^2(\mathbf{P}^n,\Lambda)$ becomes divisible by $p^m$ after pullback along $\phi^m$, this last vanishing assertion is clear for those classes in $H^j(Z,\Lambda)$ that are divisible by $c_1(\mathcal{O}(1))|_Z$, i.e., are cup product of classes in $H^{j-2}(Z,\Lambda)$ with $c_1(\mathcal{O}(1))|_Z$. However, typically this accounts for very few classes integrally\footnote{With rational coefficients, the picture is quite different, e.g., if $Z$ is smooth, then every class in $H^{> d}(Z,\mathbf{Q}_p)$ is divisible by $c_1(\mathcal{O}(1))$ by Hard Lefschetz. Thus, the non-triviality of Theorem~\ref{PnVanishingIntro} is related to the failure of Hard Lefschetz with integral coefficients.}: for instance, if $\mathcal{O}(1)|_Z$ was itself a large $p$-power as a line bundle, then $c_1(\mathcal{O}(1))|_Z \in H^2(Z,\Lambda)$ would simply be $0$, yet $H^{>d}(Z,\Lambda) \neq 0$ as long as $d > 0$.

\end{example}

Due to the interaction between the dimension of the support and cohomological degree in Theorem~\ref{PnVanishing}, it is natural to reformulate the theorem in terms of perverse sheaves: it states that the functor $\colim_i R\Gamma(\mathbf{P}^n, (\phi^i)^* (-))$ is perverse right exact. We shall obtain this statement as a reflection of the ramification properties of the map $\phi$. More precisely, we deduce it from the following general  statement about certain towers of varieties that might be useful in other similar contexts:

\begin{theorem}[Infinite ramification reduces cohomological dimension]
\label{GeneralVanThm}
Let $X/k$ be a finite type $k$-scheme  with an effective Cartier divisor $D \subset X$ with complement $U$. Let $f:Y \to X$ be an integral surjective map of $k$-schemes such that $f^* D$ admits a compatible system of $p$-power roots. Then
\[ \mathrm{pcd}(Y) = \max(\mathrm{pcd}\left(f^{-1} D), \mathrm{pcd}(f^{-1} U)\right),\]
where $\mathrm{pcd}(-)$ denotes the perverse cohomological dimension\footnote{\label{pcd}More precisely, if $Z$ is any $k$-scheme that can be obtained as an inverse limit of finite type $k$-schemes along finite transition maps (such as $Y$, $f^{-1} D$ or $f^{-1} U$), then, as discussed in \S \ref{sec:IntVar}, there is a reasonable perverse $t$-structure on $D(Z)$ that can be defined via a limiting process from the finite type case. With this definition, $\mathrm{pcd}(Z)$ is the minimal $i$ such that $R\Gamma(Z,-)$ carries ${}^p D^{\leq 0}(Z)$ into $D^{\leq i}(\Lambda)$. Note that $i \geq 0$ if $Z \neq \emptyset$ by contemplating skyscraper sheaves.} In particular, if $U$ is affine (e.g., if $D$ is ample), then 
\[ \mathrm{pcd}(Y) = \mathrm{pcd}(f^{-1} D).\]
\end{theorem}

Note that Theorem~\ref{GeneralVanThm} is essentially never true in the  finite type world for $X$ proper, e.g., if $X$ is a smooth projective curve and $D = \{x\} \subset X$ is a closed point, then $\mathrm{pcd}(D) = \mathrm{pcd}(X-D) = 0$ but $\mathrm{pcd}(X) = 1$. 
To apply Theorem~\ref{GeneralVanThm} to the tower in Theorem~\ref{PnVanishingIntro}, one takes  $D$ to be one of the co-ordinate hyperplanes in $X=\mathbf{P}^n$ and proceeds via induction.

\begin{remark}[The history of Theorem~\ref{PnVanishingIntro}]
\label{rmkhistory}
The special case of constant sheaves (Example~\ref{exsubvar}) when $k$ has characteristic $0$ was  proven by Scholze \cite[Theorem 17.3]{ScholzeICM}, using perfectoid spaces. Esnault then gave a relatively elementary proof \cite[Theorem 1.2]{EsnaultCohVan} of Theorem~\ref{PnVanishingIntro} in the generality above. In the meanwhile,  a generalization of the almost purity theorem was proven (\cite{AndreAbhyankar} and \cite[Theorem 10.9]{BhattScholzePrisms}), which allowed Reinecke \cite[Example 3.4]{ReineckeMg} to extend Scholze's argument to treat arbitrary $F$ as in Theorem~\ref{PnVanishing}, still when $k$ has characteristic $0$. 
\end{remark}

\begin{remark}[An abelian variety analog]
\label{rmkabvar}
As an abelian variety analog of  Theorem~\ref{PnVanishingIntro}, one can ask the following question (raised previously in \cite[Remark 2.11]{BSS}): if $A/k$ is an abelian variety and $F$ is a constructible $\Lambda$-sheaf on $A$, does
\[ \colim_{n} H^i(A, [p^n]^* F) = 0\]
for $i > \dim \mathrm{Supp}(F)$? Here $[p^n]:A \to A$ denotes the multiplication by $p^n$ map. When $k$ has characteristic $0$, this can be shown via perfectoid methods (see \cite[Theorem 1]{PerfectoidAbelianAWS} and \cite[Example 3.5]{ReineckeMg}); in fact, when $k$ is a perfectoid $p$-adic field, this method yields the theorem for possibly non-algebraizable abeloid varieties. Similarly, when $k=\mathbf{C}$,  the claim holds true for all compact complex tori using the results in \cite{BSS} (see Theorem~\ref{AbVarVanChar0} below). However, for ground fields of positive characteristic, the question is open in general and one only has partial results (though closely related questions do have positive answers for $\mathbf{Q}_\ell$-coefficients by recent work \cite{EsnaultKerzHL} of Esnault--Kerz).  In fact, our original motivation for pursuing an alternative proof of Theorem~\ref{PnVanishingIntro} was precisely  to obtain a proof that might adapt to the abelian variety setting, though that goal  remains just as elusive: the proof of Theorem~\ref{PnVanishingIntro} relies heavily on exploiting the ramification of $\phi$, and it is unclear how to adapt the strategy to the unramified map $[p]:A \to A$.
\end{remark}

\begin{remark}
We have assumed throughout the paper that the order of $\Lambda$ is invertible on $k$. This assumption is made to avoid trivialities and is not necessary. In fact, if $\# \Lambda$ is a power of the characteristic of $k$, then all the results (Theorem~\ref{PnVanishingIntro}, Theorem~\ref{GeneralVanThm} as well as the abelian variety analog as in Remark~\ref{rmkabvar}) are still true and rather  formal: thanks to the Artin--Schreier sequence, the $\Lambda$-cohomological dimension of any proper $k$-scheme $X$ is $\dim(X)$ in this case. 
\end{remark}

We briefly explain the proof of Theorem~\ref{GeneralVanThm}. Our main observation is the analogy of the statement of the theorem with the following classical phenomenon in \'etale sheaf theory. Given a flat morphism $X \to \mathrm{Spec}(V)$ over a dvr $V$ with generic and special fibres $j:X_\eta \subset X$ and $i:X_s \subset X$ respectively, one has the functor $i^* j_*:D(X_\eta) \to D(X_s)$ as well as its variant $\Psi:D(X_{\overline{\eta}}) \to D(X_{\overline{s}})$ (called the nearby cycles functor) for the base change $X_{\overline{V}}$ to the integral closure $\overline{V}$ of $V$ in an algebraic closure $\overline{\eta}$ of $\eta$. It is a basic fact that $\Psi$ has better cohomological properties than $i^* j_*$; for example, $\Psi$ is perverse $t$-exact (under appropriate normalizations), while $i^* j_*$ has perverse cohomological dimension $1$. Looking at the proof, this better behaviour is essentially a reflection of the infinite ramification of $X_{\overline{V}} \to X$ along the divisor $X_s \subset X$. We observe (Proposition~\ref{nearby}) that similar behaviour also occurs in the global setup of Theorem~\ref{GeneralVanThm}, which then facilitates an argument via induction on dimension to prove Theorem~\ref{PnVanishingIntro}.

\subsection*{Acknowledgements}
I am grateful to H\'el\`{e}ne Esnault, Martin Olsson, Emanuel Reinecke, Peter Scholze, Jakob Stix and the anonymous referee for their comments and encouragement.  During the preparation of this note (which dates back to the summer of 2020), I was partially supported by the NSF (\#1801689, \#1840234, \#1952399), Packard and Simons foundations.

\section{Int-varieties}
\label{sec:IntVar}

In this section, we simply observe that passage to the direct limit yields a  reasonable theory of semiperverse sheaves on schemes that are integral over finite type $k$-schemes; especially, the connective part of the perverse $t$-structure behaves almost exactly as well as the classical case.

\begin{definition}
A qcqs $k$-scheme $X$ is called a {\em int-variety} if there exists an integral $k$-morphism $X \to Y$ with $Y$ a scheme of finite type over $k$. 
\end{definition}

All int-varieties of interest to us arise as follows:

\begin{example}
Let $Y$ be an integral $k$-variety, and let $L/K(Y)$ be an algebraic extension of fields. Then the normalization $X$ of $Y$ in $L$ is an int-variety. 
\end{example}

In the language of \cite{HamacherCompactSupport}, int-varieties are exactly the $k$-schemes of finite expansion. 

\begin{lemma}
\label{rmk:intpres}
Say $X$ is an int-variety. Then we can write $X = \lim_i X_i$ as a cofiltered inverse limit of finite type $k$-schemes with the transition maps being finite surjective. Moreover, for any such presentation, we have $\dim(X) = \dim(X_i)$ for all $i$. 
\end{lemma}

We shall call such a presentation an {\em int-presentation} of $X$. Note that all presentations of $X$ as a cofiltered limit of finite type $k$-schemes are pro-equivalent to each other.

\begin{proof}
For the first part, say $f:X \to Y$ is an integral map with $Y$ a finite type $k$-scheme. Consider the quasicoherent sheaf of algebras $\mathcal{A} := f_* \mathcal{O}_X$ on $Y$, so $X=\mathrm{Spec}_Y(\mathcal{A})$. As $Y$ is noetherian and $\mathcal{A}$ is integral over $\mathcal{O}_Y$, we can write $\mathcal{A} = \cup_i \mathcal{A}_i$ as a filtering union of its $\mathcal{O}_Y$-coherent subalgebras. Setting $X_i = \mathrm{Spec}_Y(\mathcal{A}_i)$ gives the desired presentation.

For the second part, it suffices to show that if $f:X \to Y$ is a surjective integral map with $Y$ a $k$-scheme of finite type, then we have $\dim(X) = \dim(Y)$: indeed, the inequality $\leq$ follows as there cannot be any non-trivial specializations in the fibres of $f$, while the inequality $\geq$ follows from the going up theorem.
\end{proof}

\begin{remark}[\'Etale cohomological dimension of an int-variety]
\label{rmk:cohdim}
Say $X$ is an int-variety of Krull dimension $d$. Then $\Lambda$-cohomological dimension of $U_{et}$ for any affine open $U \subset X$ is $\leq d$. Indeed, by Lemma~\ref{rmk:intpres} and noetherian approximation, such a $U$ admits an int-presentation $U = \lim_i U_i$ with each $U_i$ being an affine $k$-scheme of finite type of dimension $\leq d$; the claim then follows from Artin vanishing on each $U_i$ coupled with the observation that any constructible sheaf $F$ of $\Lambda$-modules on $U$ is pulled back from some $U_i$. A similar argument shows that the $\Lambda$-cohomological dimension of $X$ itself is $\leq 2d$.
\end{remark}

\begin{construction}[Semiperverse complexes on int-varieties]
\label{SemiPervComp}
Any int-variety $X$ has finite Krull dimension by Lemma~\ref{rmk:intpres}, so one has a well-defined function $p_X:X \to \mathbf{Z}$ given by $p_X(x) = -\dim(\overline{\{x\}})$. In fact, given an integral map $f:X \to Y$, we have $p_X = p_Y \circ f$: this amounts to observing that if $A \subset B$ is an integral extension of integral domains, then $\dim(A) = \dim(B)$. This observation implies that $p_X$ is a strong perversity function in the sense of \cite[\S 1]{Gabbertstructures}, so one has an induced perverse $t$-structure on $D(X)$ by \cite[\S 6]{Gabbertstructures}. Explicitly, if $X_{geom}$ denotes the set of isomorphism classes of geometric points $x \to X$, then we have
\[ {}^p D^{\leq 0}(X) = \{K \in D(X) \mid K_x \in D^{\leq p_X(x)} \ \forall x \in X_{geom}\} \]
and 
\[ {}^p D^{\geq 0}(X) = \{K \in D(X) \mid K \in D^+, \ R\Gamma_x(K_x) \in D^{\geq p_X(x)} \ \forall x \in X_{geom} \}, \]
where $K_x$ denotes the stalk at $x$ and $R\Gamma_x(K_x) := R\Gamma_x(\mathrm{Spec}(\mathcal{O}_{X,x}^{sh}), K_x)$ denotes the local \'etale cohomology (or the ``costalk'') at $x$. We shall refer to ${}^p D^{\leq 0}(X) \subset D(X)$ as the full subcategory of {\em semiperverse} complexes. Functors that preserve this subcategory are called perverse right $t$-exact.
\end{construction}

\begin{remark}
In this paper, we shall only use the notion of semiperverse complexes, and not the full strength of the perverse $t$-structure.
\end{remark}

Semiperverse complexes have some obvious stability properties, e.g., they are preserved by pullback along open/closed immersions as well as pushforward along closed immersions. In our eventual application, we shall need the following int-variant of the standard stability property of semiperversity under finite morphisms:

\begin{lemma}[Semiperversity and integral morphisms]
\label{PushSemiperverse}
Let $f:X \to Y$ be an integral morphism of int-varieties.
\begin{enumerate}
\item $f_*:D(X) \to D(Y)$ preserves semiperversity.
\item $f^*:D(Y) \to D(X)$ preserves semiperversity.
\item $f_*$ detects semiperversity, i.e., given $K \in D(X)$, if $f_* K \in {}^p D^{\leq 0}$, then $K \in {}^p D^{\leq 0}$.
\item If $f$ is surjective, then $f^*$ detects semiperversity. 
\end{enumerate}
\end{lemma}
\begin{proof}
Let us first make some general remarks. As $f$ is integral, we have $p_X = p_Y \circ f$ as in Construction~\ref{SemiPervComp}. Moreover, for any geometric point $y \to Y$, the qcqs scheme $\widetilde{X_y} := X \times_Y \mathrm{Spec}(\mathcal{O}_{Y,y}^{sh})$, being integral over the strictly henselian local scheme $\widetilde{Y_y} := \mathrm{Spec}(\mathcal{O}_{Y,y}^{sh})$, is acyclic (i.e., each connected component is a strictly henselian local scheme) and its profinite set of connected components identifies with its set of closed points which identifies with the profinite set $S := |f^{-1}(y)|$. Write $g:\widetilde{X_y} \to S$ for the natural projection onto the profinite set of connected components. In these terms, the functor $K \mapsto (f_* K)_y$ on $D(X)$ identifies with $K \mapsto R\Gamma(S, g_* (K|_{\widetilde{X_y}}))$; furthermore, if $x \in X$  is a geometric point above $y$, then $x$ can be regarded as a point of $S$, and the functor $K \mapsto K_x$ on $D(X)$ is identified with the functor $K \mapsto (g_* (K|_{\widetilde{X_y}}))_x$, 

\begin{enumerate}
\item Fix $K \in {}^p D^{\leq 0}(X)$ and a geometric point $y \to Y$. Write $p = p_Y(y)$. We must show that $(f_* K)_y \in D^{\leq p}$. As above, we can identify $(f_* K)_y = R\Gamma(S, g_* (K|_{\widetilde{X_y}}))$; we must show this lies in $D^{\leq p}$.  For any geometric point $x \to X$ above $y$, we have $K_x \in D^{\leq p}$ since $p_X = p_Y \circ f$. By the remarks above, we then have $g_*(K|_{\widetilde{X_y}}) \in D^{\leq p}(S)$ as well. The desired claim now follows from the fact that $R\Gamma(S,-)$ is $t$-exact as any open cover of $S$ has a refinement by a clopen cover. 

\item This follows by noting that $(f^* K)_x= K_{f(x)}$ and $p_Y(f(x)) = p_X(x)$ for any geometric point $x \to X$. 

\item Fix a geometric point $x \to X$ and $K \in D(X)$ such that $f_* K \in {}^p D^{\leq 0}(Y)$. We must show that $K_x \in D^{\leq p_X(x)}$. Write $y \to Y$ for the image of $x$ and let $M := g_* (K|_{\widetilde{X_y}})$ be as in the remarks above.  We then know that $K_x$ is a stalk of $M$ while $(f_* K)_y = R\Gamma(S,M)$. Also, we have $p_Y(y) = p_X(x)$. As $S$ is a profinite set, the functor $R\Gamma(S,-)$ is exact as in (1).  It is thus enough to show that if $F$ is a sheaf of abelian groups on $S$ with $F(S) = 0$, then $F = 0$. But if $F \neq 0$, then there exists some clopen $U \subset S$ with $F(U )\neq 0$. Since $F(S) = F(U) \times F(S-U)$, we learn that $F(S) \neq 0$, as wanted. 

\item This follows by the reasoning in (2), as every geometric point $y \to Y$ has a lift to $X$.  \qedhere
\end{enumerate}
\end{proof}

\begin{remark}[Approximating constructible and semiperverse sheaves on int-varieties]
For psychological comfort, we recall a direct description of $D(X)$ and ${}^p D^{\leq 0}(X)$ for an int-variety $X$ in terms of finite type $k$-schemes. For the former, note that any qcqs $U \in X_{et}$ has finite \'etale cohomological dimension by Artin's theorem (Remark~\ref{rmk:cohdim}), so taking the colimit induces an equivalence
\[ \colim: \mathrm{Ind}(D^b_{cons}(X,\Lambda)) \simeq D(X)\]
by \cite[\S 6.5]{BhattScholzeProetale}.
Furthermore, given an int-presentation $X = \lim_i X_i$, we have $D^b_{cons}(X) = \colim_i D^b_{cons}(X_i)$, whence there is an induced equivalence 
\[ \mathrm{Ind}(\colim D^b_{cons}(X_i)) \simeq \mathrm{Ind}(D^b_{cons}(X)) \simeq D(X)\]
with the composite given by taking the colimit. We claim that this restricts to an equivalence 
\[ \mathrm{Ind}(\colim_i {}^p D^{\leq 0}_{cons}(X_i)) \simeq {}^p D^{\leq 0}(X).\]
As semiperversity is preserved under pullback along integral morphisms, we have a fully faithful map from the left to the right in the displayed formula above. To show this an equivalence, fix $F \in {}^p D^{\leq 0}(X)$. If $f_i:X \to X_i$ denotes the projection, then we have a natural equivalence 
\[ \colim_i f_i^* f_{i,*} F \to F.\] 
It is thus enough to note that $f_{i,*} F \in {}^p D^{\leq 0}(X_i)$ by Lemma~\ref{PushSemiperverse} (1).
\end{remark}

\section{Infinitely ramified maps}

In this section, we introduce the notion of an infinitely ramified map. While one could define this notion in an elementary fashion, we have chosen a formulation in terms of infinite root stacks for conciseness and later convenience; this also allows us to obtain statements close to the language of log geometry without explicitly introducing the latter (see \cite{TalpoVistoli} for a more substantial discussion).

\subsection{Infinite root stacks}

In this subsection, we recall the notion of an infinite root stacks of a log scheme in the special case where the log structure is generated by a single divisor. Instead of using the language of log structures, we formulate statements in terms of generalized Cartier divisors; this generalization of the notion of an effective Cartier divisor has better stability properties and has been studied a number of times in the literature, perhaps under slightly different names such as virtual Cartier divisors, pseudodivisors, or just ``divisors'' (with the quotes being part of the name); see \cite[\S 3]{FultonIntersectionTheory} \cite[\S 3]{KhanRydhVirtualCar} for more.

\begin{definition}[Generalized Cartier divisor]
Let $X$ be a $k$-stack. A {\em generalized Cartier divisor} on $X$ consists of a pair $(L,s)$, where $L$ is an invertible $\mathcal{O}_X$-module and $s:L \to \mathcal{O}_X$ is an $\mathcal{O}_X$-linear map. Write $\mathrm{GCart}(X)$ for the groupoid of generalized Cartier divisors; note that the full subgroupoid $\mathrm{Cart}(X) \subset \mathrm{GCart}(X)$ consisting of pairs $(L,s)$ with $s$ injective identifies with the set of effective Cartier divisors on $X$. 
\end{definition}

When $X$ is integral, which is the case for our applications, any $(L,s) \in \mathrm{GCart}(X)$ with $s \neq 0$ is effective. However, generalized Cartier divisors have the advantage of being stable under arbitrary pullbacks (unlike classical Cartier divisors), so we can regard $\mathrm{GCart}(-)$ as a presheaf on $k$-stacks. 

\begin{remark}[The universal generalized Cartier divisor]
\label{rmk:univ}
The presheaf $\mathrm{GCart}(-)$ is representable by $\mathbf{A}^1/\mathbf{G}_m$, with the universal generalized Cartier divisor corresponding to the tautological map $\mathcal{O}(-1) \to \mathcal{O}$ on $\mathbf{A}^1/\mathbf{G}_m$. This follows by unwinding the definition of $\mathbf{A}^1/\mathbf{G}_m$ in terms of $\mathbf{G}_m$-torsors equipped with a function.
\end{remark}

\begin{remark}[The monoid structure on $\mathrm{GCart}(-)$]
There is a natural multiplication on $\mathrm{GCart}(-)$, given on local sections by $(L,s) \otimes (M,t) = (L \otimes M, s \otimes t)$. In fact, this corresponds to the commutative monoid stack structure on $\mathbf{A}^1/\mathbf{G}_m$ determined by multiplication on $\mathbf{A}^1$ and $\mathbf{G}_m$. In particular, it makes sense to ask for an $n$-th root of a point $(L,s) \in \mathrm{GCart}(S)$: it is a lift of $(L,s)$ along the $n$-power map $\mathrm{GCart}(S) \to \mathrm{GCart}(S)$. 
\end{remark}

Taking the inverse limit of the multiplication by $p^n$ maps on the monoid $\mathrm{GCart}(-)$ then leads to the following construction:

\begin{construction}[Infinite root stacks]
Let $X$ be a stack on $k$-algebras equipped with a generalized Cartier divisor $D := (L,s) \in \mathrm{GCart}(X)$. By Remark~\ref{rmk:univ}, the pair $(L,s)$ is determined by its classifying map $X \to \mathbf{A}^1/\mathbf{G}_m$. Set $\widetilde{\mathbf{A}^1} := \lim_{x \mapsto x^p} \mathbf{A}^1$ and $\widetilde{\mathbf{G}_m} := \lim_{x \mapsto x^p} \mathbf{G}_m$, and define the stack $X[D^{1/p^\infty}]$ via the fibre product diagram
\begin{equation}
\label{infroot}
 \xymatrix{ X[D^{1/p^\infty}] \ar[r] \ar[d] & \widetilde{\mathbf{A}^1}/\widetilde{\mathbf{G}_m} = \lim_p \mathrm{GCart}(-) \ar[d] \\
X \ar[r] &   \mathbf{A}^1/\mathbf{G}_m = \mathrm{GCart}(-).}
\end{equation}
Unwinding definitions, one learns that the left vertical map is the universal stack over $X$ equipped with a compatible system $\{L_n\}_{n \geq 1}$ of $p$-power roots of $L_0 := \pi^* L$ as well as a compatible system $\{s_n \in H^0(X[D^{1/p^\infty}],L_n^{-1})\}_{n \geq 1}$ of $p$-power roots of the canonical section $s_0  \in H^0(X[D^{1/p^\infty}], L_0^{-1})$ defining the pullback of $D$. Some features of this construction are:

\begin{enumerate}
\item Base change: The construction $(X,D) \mapsto X[D^{1/\infty}]$ commutes with base change on $X$, i.e., if $f:Y \to X$ is a map and $D_Y = f^* D \in \mathrm{GCart}(Y)$ is the pullback of $D$, then there is a natural identification $Y[D_Y^{1/p^\infty}] \simeq X[D^{1/p^\infty}] \times_X Y$.

\item Flatness:  The structure map $X[D^{1/p^\infty}] \to X$ is faithfully flat. In particular, if $(L,s)$ is an effective Cartier divisor, the same holds true for the $p^n$-th root $(L_n,s_n) \in \mathrm{GCar}(X[D^{1/p^\infty}])$.
\end{enumerate}
\end{construction}

\begin{remark}[The infinite root stack as a modification]
Let $X$ be a stack on $k$-algebras equipped with a generalized Cartier divisor $D := (L,s) \in \mathrm{GCart}(X)$. If the section $s$ is invertible, then $X[D^{1/p^\infty}] \to X$ is an isomorphism: the right vertical map in diagram \eqref{infroot} is an isomorphism over $\mathbf{G}_m/\mathbf{G}_m \subset \mathbf{A}^1/\mathbf{G}_m$. On the other hand, if $s=0$, then $X[D^{1/p^\infty}] \to X$ can be identified with a locally nilpotent thickening of the $\mathbf{Z}_p(1)$-gerbe $X' \to X$ of $p$-power roots of the line bundle $L$ (also defined as the image of $L$ under the boundary map of the extension
\[ \mathbf{Z}_p(1) \to \widetilde{\mathbf{G}_m} \to \mathbf{G}_m\]
of group schemes). Thus, in general, we can view $X[D^{1/p^\infty}] \to X$ as a modification of $X$, which is an isomorphism over $U=X-V(s)$ and has (reduced) fibres $B\mathbf{Z}_p(1)$ over $V(s)$. Thus, it can roughly be viewed as an algebro-geometric analog of the real-analytic blowup of $X$ along $D$ when $X$ is a complex variety and $D$ is effective.
\end{remark}

\begin{example}[The case of a principal divisor]
\label{rootstackkummertorsor}
Let $X=\mathbf{A}^1 = \mathrm{Spec}(k[x])$  equipped with the effective Cartier divisor $V(x) \subset X$, corresponding to $(\mathcal{O},x) \in \mathrm{GCart}(X)$. Then $X[D^{1/p^\infty}]$ is identified with $\widetilde{\mathbf{A}^1}/\mathbf{Z}_p(1)$, where $\widetilde{\mathbf{A}^1} = \lim_{x \mapsto x^p} \mathbf{A}^1 = \mathrm{Spec}(k[x^{1/p^\infty}])$ as above and $\mathbf{Z}_p(1) := \lim_n \mu_{p^n} = \ker(\widetilde{\mathbf{G}_m} \to \mathbf{G}_m)$. The resulting  map $X[D^{1/p^\infty}] \to B\mathbf{Z}_p(1)$ classifies a $\mathbf{Z}_p(1)$-torsor $T \to X[D^{1/p^\infty}]$. Over the open subset $U=X-V(x) \subset X[D^{1/p^\infty}]$, the torsor $T|_U \to U$ is simply the Kummer torsor associated to the invertible function $x \in \mathcal{O}(U)$.

More generally, given any $k$-scheme $X$ with a function $f:X \to \mathbf{A}^1$ corresponding to $D=(\mathcal{O},f) \in \mathrm{GCart}(X)$,  a similar analysis shows that $X[D^{1/p^\infty}]$ comes equipped with canonical a $\mathbf{Z}_p(1)$-torsor extending the Kummer torsor for the invertible function $f$ on $U=X-V(f)$. This extension is  a stacky version of (an instance of) Abhyankar's lemma: finite \'etale covers of $p$-power order over $U$ extend to finite \'etale covers of $X$ after adding $p$-power ramification along $V(f)$.
\end{example}

\begin{remark}
Given a generalized Cartier divisor $D := (L,s) \in \mathrm{GCart}(X)$, we obtain a closed subscheme $V(s) \subset X$; we shall sometimes also denote the latter by $D$ if no confusion arises. Note that if $D$ is effective, we can recover $(L,s)$ from $V(s)$, but not in general. In particular, both the closed subschemes $\emptyset \subset X$ and $X \subset X$ arise from this construction applied to $(\mathcal{O}_X,1)$ and $(L,0)$ (for any line bundle $L$) respectively. 
\end{remark}

\subsection{Infinitely ramified maps}

In this subsection, we study the following notion:

\begin{definition}[Infinitely ramified maps]
Let $X$ be a scheme equipped with a generalized Cartier divisor $D$. An integral map $f:Y \to X$ is said to be {\em infinitely $p$-ramified along $D$} if there exists a factorization of $f$ through the structure map $X[D^{1/p^\infty}] \to X$ (or equivalently that $f^* D \in \mathrm{GCart}(Y)$ admits a compatible system of $p$-power roots). 
\end{definition}

There are very few infinitely ramified maps between finite type $k$-schemes.

\begin{example}[The finite type case]
Let $X$ be finite type $k$-scheme equipped with an effective Cartier divisor $D$. Let $f:Y \to X$ be an integral (or equivalently finite) map of finite type $k$-schemes. Then $f$ is infinitely ramified along $D$ exactly when $f^* D$ equals $0$ or $Y$ as a generalized Cartier divisor on $Y$. Indeed, passing to local rings of $Y$, this amounts to the following observation (proven using Krull's intersection theorem): if $R$ is a noetherian local ring and $f \in R$ such that $(f)$ admits $n$-th roots for infinitely many integers $n$, then either $f=0$ or $f$ is a unit.
\end{example}

Thus, to obtain interesting examples, one must leave the world of finite type $k$-schemes. The basic local phenomenon is:

\begin{example}[$p$-power roots of functions]
Let $X$ be a $k$-scheme equipped with a function $f \in \mathcal{O}(X)$. Let $\pi:Y \to X$ be an integral map of $k$-schemes. If $\pi^*f$ admits a compatible system of $p$-power roots, then $\pi$ is infinitely $p$-ramified along $D=(\mathcal{O},f) \in \mathrm{GCart}(X)$.
\end{example}

Globally, one has the following examples, with the first one being most relevant to this note:

\begin{example}[The toric Frobenius tower]
Let $X=\mathbf{P}^n$ and let $D = \mathbf{P}^{n-1} \subset X$ be a co-ordinate hyperplane. Write $\phi:X \to X$ for the standard Frobenius lift, given by $[x_0,...,x_n] \mapsto [x_0^p,...,x_n^p]$ in homogeneous co-ordinates. Note that the pullback $\phi^* D$ admits a natural $p$-th root as an effective Cartier divisor. Iterating this observation shows that the projection map $\lim_\phi X \to X$ is infinitely $p$-ramified along $D$.  (More generally, there is a natural analog of this construction for any toric variety, but we do not spell it out here.)
\end{example}

\begin{example}[The level structure tower for the moduli of abelian varieties]
Fix a number $g \geq 1$.
Let $\mathcal{A}_g$ be the moduli stack of principally polarized abelian varieties of dimension $g$ and let $\mathcal{A}_g \subset \overline{\mathcal{A}_g}$ be a toroidal compactification as in \cite[\S 4]{FaltingsChai}. For each integer $n \geq 1$, we have a finite \'etale cover $\mathcal{A}_g[p^n] \to \mathcal{A}_g$ obtained by trivializing the $p^n$-torsion of the abelian variety.  Let $\overline{\mathcal{A}_g}[p^n]$ be the normalization of $\overline{\mathcal{A}_g}$ in this cover. Then $\{\overline{\mathcal{A}_g}[p^n]\}$ is a tower of finite surjective morphisms over $\overline{\mathcal{A}_g}$ that is \'etale over $\mathcal{A}_g$. On the other hand, if $D = \overline{\mathcal{A}_g} - \mathcal{A}_g$ denotes the complement, then $D$ is a divisor and this tower is highly ramified along $D$; in fact, by \cite[Theorem IV 6.7 (6)]{FaltingsChai}, the map $\lim_n \overline{\mathcal{A}_g}[p^n] \to \overline{\mathcal{A}_g}$ is infinitely $p$-ramified along each component of $D$.
\end{example}

\begin{example}[The level structure tower for the moduli of curves]
Fix a number $g \geq 1$.
Let $\mathcal{M}_g$ be the moduli stack of smooth genus $g$ curves, and let $\mathcal{M}_g \subset \overline{\mathcal{M}_g}$ be the compactification provided by the stack of stable genus $g$ curves. For each integer $n \geq 1$, we have a finite \'etale cover $\mathcal{M}_g[p^n] \to \mathcal{M}_g$ obtained by trivializing the $p^n$-torsion of the Jacobian.  Let $\overline{\mathcal{M}_g}[p^n]$ be the normalization of $\overline{\mathcal{M}_g}$ in this cover. Then $\{\overline{\mathcal{M}_g}[p^n]\}$ is a tower of finite surjective morphisms over $\overline{\mathcal{M}_g}$ that is \'etale over the open substack $\mathcal{M}_g^c \subset \overline{\mathcal{M}_g}$ of compact type stable curves. On the other hand, if $D = \overline{\mathcal{M}_g} - \overline{\mathcal{M}_g}^c$ denotes the complement, then $D$ is a divisor and this tower is highly ramified along $D$; in fact, by \cite[Lemma 5.11]{ReineckeMg}, the map $\lim_n \overline{\mathcal{M}_g}[p^n] \to \overline{\mathcal{M}_g}$ is infinitely $p$-ramified along each component of $D$.
\end{example}

\begin{example}[The absolute integral closure]
Let $X/k$ be a integral variety.. Let $f:X^+ \to X$ be its absolute integral closure, i.e., the integral closure of $X$ in an algebraic closure of its function field. For any $D \in \mathrm{GCart}(X)$, we claim that the map $f$ is infinitely $p$-ramified along $D$. It suffices to show that for one can find a finite cover of $X$ (which can always be dominated by $f$) where $D$ acquires a $p$-th root as an effective Cartier divisor; iterating this construction and pulling back the resulting data to $X^+$ then yields the desired compatible system of $p$-power roots of $D$. To find these roots, we first run the Bloch--Gieseker construction\footnote{The reference \cite{LazarsfeldPos1} assumes $X$ is quasi-projective to produce $n$-th roots of line bundles on finite covers. However, this is not necessary: the Picard group of $X^+$ is uniquely divisible by \cite[Lemma 6.6]{BhattCMRPlus}, so any line bundle on a $k$-variety acquires an $p$-th root over a sufficiently large finite cover.} to find a $p$-th root $L$ of $\mathcal{O}(D)$ over a finite cover, and then  the cyclic covering trick to produce a $p$-th root in $L$ of the defining section of $D$ after a further finite cover;  see \cite[\S 4.1.B]{LazarsfeldPos1} for both these tricks.
\end{example}

\section{Ramification implies vanishing}

In this section, we prove our main vanishing result for infinitely ramified maps.

\begin{notation}
Let $X$ be a finite type $k$-scheme equipped with an effective Cartier divisor $D$ with complement $U = X-D$. Let $f:Y \to X$ be an integral map that is infinitely $p$-ramified along $D$ (so $Y$ is an int-scheme). Write $i_Y:f^{-1}D \subset Y$ and $j_Y:f^{-1}U \subset Y$ for the natural inclusions. Note that $Y$, $f^{-1} U$ and $f^{-1} D$ are all int-varieties, so they support a well-behaved notion of semiperverse sheaves by the material in \S \ref{sec:IntVar}.
\end{notation}

Our main observation is the following global variant of semiperversity of nearby cycles:

\begin{proposition}[A consequence of the semiperversity of nearby cycles]
\label{nearby}
The functor
\[ i_Y^* \circ j_{Y,*}[-1]: D(f^{-1}U) \to D(f^{-1}D)\]
preserves semiperversity.
\end{proposition}

\begin{remark}
The assertion in Proposition~\ref{nearby} is clearly false without infinite $p$-ramification along $D$. For example, if $X=\mathbf{A}^1$ with $D$ being the origin, the functor $i^* j_*[-1]:D(U) \to D(D) \simeq D(\Lambda)$ carries $\Lambda[1]$ to the complex $R\Gamma(S^1,\Lambda) = \Lambda \oplus \Lambda[-1]$, which is not semiperverse on $D=\mathrm{Spec}(k)$. The effect of adding infinite ramification along $D$ is to kill the $H^1$ group appearing here.
\end{remark}

\begin{proof}
The statement is local on $X$, so we may assume that there exists a map $g:X \to \mathbf{A}^1 := \mathrm{Spec}(k[t])$ with $g^{-1}(0) = D$. Let $\widetilde{\mathbf{A}^1} := \mathrm{Spec}(k[t^{1/p^\infty}]) \to \mathbf{A}^1$ be mock perfection of $\mathbf{A}^1$. For any scheme $W/\mathbf{A}^1$, write $\widetilde{W} := W \times_{\mathbf{A}^1} \widetilde{\mathbf{A}^1} \to W$ denote the base change, so $\widetilde{\mathbf{G}_m} \to \mathbf{G}_m$ is the $\mathbf{Z}_p(1)$-torsor of $p$-power roots of $t$ (Remark~\ref{rootstackkummertorsor}). Thus, we obtain the diagram
\[ \xymatrix{ 
		  \widetilde{D} \ar[r]^-{i_{\widetilde{X}}} \ar[d]^-{\pi_D} & \widetilde{X} \ar[d]^-{\pi} & \widetilde{U} \ar[l]_{j_{\widetilde{X}}} \ar[d]^-{\pi_U} \\
		  D \ar[r]^-{i_X} \ar[d]^-{g_D} & X \ar[d]^-g & U \ar[l]_-{j_X} \ar[d]^-{g_U} \\
		  \{0\} \ar[r] & \mathbf{A}^1 & \mathbf{G}_m \ar[l] }\]
where all squares are cartesian and $\pi$ is integral. We shall first prove the proposition for $Y = \widetilde{X}$, and then make a series of reductions to this case.

Consider first the case $Y = \widetilde{X}$ and $f=\pi$. We must show that $i_{\widetilde{X}}^* \circ j_{\widetilde{X},*}[-1]:D(\widetilde{U}) \to D(\widetilde{D})$ preserves semiperversity. This assertion is local along $D$, so we can replace $\mathbf{A}^1$ with its henselization at $0$ (and pullback all objects appearing above to this henselization) to check the assertion. In this case, it  follows essentially from the classical theory of nearby cycles (see \cite[Proposition 4.4.2]{BBDG}, specifically \cite[4.4.4]{BBDG}); for completeness, we also review the argument in Proposition~\ref{nearbysemi} below.

Next, assume there exists an integral $X$-morphism $Y \to \widetilde{X}$. In this case, one can deduce the result for $Y$ from the result for $\widetilde{X}$ using proper base change and Lemma~\ref{PushSemiperverse} (1), (2) and (3).

Finally, we tackle the general case. By assumption on $Y \to X$, there exists an $X$-map 
\[ Y \to X[D^{1/p^\infty}] := X \times_{\mathbf{A}^l} \widetilde{\mathbf{A}^1}/\mathbf{Z}_p(1) \simeq X \times_{\mathbf{A}^1/\mathbf{G}_m} \widetilde{\mathbf{A}^1}/\widetilde{\mathbf{G}_m}.\] 
Thus, there is a $\mathbf{Z}_p(1)$-torsor $Y' \to Y$ equipped with an integral $X$-map $Y' \to \widetilde{X}$. By the previous reductions, we know the statement in question for $Y'$. We have a commutative diagram of $X$-schemes of the form
\[ \xymatrix{ D' \ar[r] \ar[d] & Y' \ar[d] & U' \ar[l] \ar[d] \\
		f^{-1} D \ar[r] & Y & f^{-1}U \ar[l] }\]
where all squares are cartesian, and the vertical maps are pro-(finite \'etale) covers. It is now easy to deduce the claim for $Y$ from that for $Y'$ using Lemma~\ref{PushSemiperverse} (4) as well as pro-smooth base change along $Y' \to Y$.
\end{proof}

\begin{corollary}
\label{nearby!restricti}
The functor 
\[ i_Y^!:D(Y) \to D(f^{-1}D)\]
preserves semiperversity. 
\end{corollary}
\begin{proof}
For any $F \in D(Y)$, we have an exact triangle
\[ i_{Y,*} i_Y^! F \to F \to j_{Y,*} j_Y^* F.\]
Now assume $F \in {}^p D^{\leq 0}(Y)$. Applying $i_Y^*$ and rotating the above gives
\[ i_Y^* j_{Y,*} j_Y^* F[-1] \to i_Y^! F \to i_Y^* F.\]
The outer terms lie in ${}^p D^{\leq 0}(D)$: this is clear from the definition for the right side and follows from Proposition~\ref{nearby} for the left side. This implies that $i_Y^! F$ is semiperverse, as wanted.
\end{proof}

Using the above, we can now prove our main theorem:

%
%
%
%

\begin{proof}[Proof of Theorem~\ref{GeneralVanThm}]
First observe that the hypothesis in Theorem~\ref{GeneralVanThm} precisely ensures that $f$ is infinitely $p$-ramified over $D$, so we are in the setup of the current section. Moreover, we always have
\[ \mathrm{pcd}(Y) \geq \max(\mathrm{pcd}\left(f^{-1} D), \mathrm{pcd}(f^{-1} U)\right),\]
since the $*$-extension of a perverse sheaf on $f^{-1} D$ (resp. $f^{-1} U$) is a perverse sheaf on $Y$ for trivial reasons (resp. by Artin vanishing as $j_Y$ is affine). For the reverse inequality, given $F \in D(Y)$, we have an exact triangle
\[ i_{Y,*} i_Y^! F \to F \to j_{Y,*} j_Y^* F.\]
Now assume $F \in {}^p D^{\leq 0}(Y)$. Applying $R\Gamma(Y,-)$ to this triangle gives
\[ R\Gamma(f^{-1} D, i_Y^! F) \to R\Gamma(Y,F) \to R\Gamma(f^{-1} U, j_Y^* F).\]
By Corollary~\ref{nearby!restricti}, the functor $i_Y^!$ preserves ${}^p D^{\leq 0}$, so the above triangle yields  the desired inequality
\[ \mathrm{pcd}(Y) \leq \max(\mathrm{pcd}\left(f^{-1} D), \mathrm{pcd}(f^{-1} U)\right).\]
The last assertion simply follows as $\mathrm{pcd}(Y) \geq 0$ (Footnote~\ref{pcd} and Remark~\ref{rmk:cohdim}) and $\mathrm{pcd}(f^{-1} U) = 0$ when $U$ is affine (Remark~\ref{rmk:cohdim} and Artin vanishing).
\end{proof}

We can then deduce the desired statement for projective space:

\begin{theorem}
\label{PnVanishing}
Let $f:Y := \lim_{\phi} \mathbf{P}^n \to X = \mathbf{P}^n$ be the mock perfection of $\mathbf{P}^n$, where $\phi([x_0,...,x_n]) = [x_0^p,....,x_n^p]$. Then $R\Gamma(Y,-)$ is perverse right $t$-exact. 
\end{theorem}

\begin{proof}
We prove the claim by induction on $n$. The $n=0$ is case is clear, so assume $n \geq 1$. Let $D = \mathbf{P}^{n-1} \subset X$ be the hyperplane at infinity (corresponding to $x_n=0$) with complement $j_X:U = \mathbf{A}^n \subset X$. The map $f:Y \to X$ and the divisor $D \subset X$ satisfy the hypothesis of Theorem~\ref{GeneralVanThm}. As $\mathrm{pcd}(f^{-1} U) = 0$ (since $U$ is affine, see Remark~\ref{rmk:cohdim}), it suffices to show that $\mathrm{pcd}(f^{-1}D) = 0$; this follows by induction since $(f^{-1}D)_{red} \to D$ is the mock perfection for $\mathbf{P}^{n-1}$.
\end{proof}

Note that Theorem~\ref{PnVanishing} implies Theorem~\ref{PnVanishingIntro}: apply the former to the pullback of $F$ to $Y$. In fact, as any semiperverse complex on $Y$ can be built via colimits by this construction (possibly applied to a different copy of $X$ appearing in the limit defining $Y$), the two theorems are equivalent. 

The essential input in the proofs above was the semiperversity of nearby cycles, \cite[Proposition 4.4.2]{BBDG}. For completeness, we review the argument below.

\begin{proposition}[Semiperversity of vanishing cycles]
\label{nearbysemi}
Let $V$ be a henselian discrete valuation ring over $k$ with residue field $k$, fraction field $K$, and uniformizer $t \in V$. Let $K_\infty = K[t^{1/p^\infty}]$. Let $X/V$ be a flat finitely presented $V$-scheme with special and generic fibres $X_s$ and $X_\eta$. Fix a perverse sheaf $F$ on $X_\eta$.  Then for any  geometric point $x \to X_s$, we have
\[ R\Gamma(X^{sh}_x \otimes_V K_\infty, F) \in D^{\leq p(x)},\]
where $p(x) = -\dim(\overline{\{x\}})$.
\end{proposition}

Note that $R\Gamma(X^{sh}_x \otimes_V K_\infty, F)$ can be identified with the stalk at $x$ of $i^* j'_* F$, where $j':X \otimes_V K_\infty \to X$ and $i:X_s \to X$ are the natural maps. 

\begin{proof}
By a standard devissage, we may assume that $X$ is integral and $F$ is the (shifted) constant sheaf, in which case we must show that
\[ R\Gamma(X^{sh}_x \otimes_V K_\infty, \Lambda) \in D^{\leq d_x},\]
where $d_x = \dim(X^{sh}_{s,x}) = \dim(X^{sh}_x) -1 = \dim(X_x^{sh} \otimes_V K)$.

We first prove the statement when $x \to X_s$ maps to a closed point, so $d_x = \dim(X_K)$. As $X_x^{sh} \otimes_V \overline{K}$ is a cofiltered limit of finitely presented affine $\overline{K}$-schemes of dimension $\leq d_x$,  Artin vanishing over $\overline{K}$ implies that $R\Gamma(X_x^{sh} \otimes_V \overline{K}, \Lambda) \in D^{\leq d_x}$. But the Galois group $G$ of $\overline{K}/K_\infty$ is a prime-to-$p$ profinite group, so the functor $\Gamma(G,-) = (-)^G$ is exact on discrete $\Lambda$-modules with a $G$-action. It follows that 
\[ R\Gamma(X_x^{sh} \otimes_V K_\infty, \Lambda) = R\Gamma(G,R\Gamma(X_x^{sh} \otimes_V \overline{K}, \Lambda)) \]
also lies in $D^{\leq d_x}$, as wanted. 

For general geometric points $x \to X_s$, we can reduce to the case of closed points by enlarging the residue field of the dvr $V$ to match $\kappa(\mathcal{O}_{X,x}^{sh})$: there exists a uniformizer preserving extension $V \to W$ of strictly henselian dvrs and a factorization $V \to W \to \mathcal{O}_{X,x}^{sh}$, such that the $W$-algebra $\mathcal{O}_{X,x}^{sh}$ is also the strict henselization $\mathcal{O}_{Y,y}^{sh}$ of a finitely presented flat integral $W$-scheme $Y$ at a closed point $y$ of the special fibre $Y_s$ (see \cite[\S 4.4.5]{BBDG} or \cite[Lemma 4.8]{BhattCMRPlus}). We can then replace $(V,X,x)$ with $(W,Y,y)$ to reduce to the case treated in the previous paragraph.
\end{proof}

\section{The abelian variety analog in characteristic $0$}

In this section, we record an argument\footnote{This argument should have appeared in \cite{BSS}: indeed, it was implicitly used to justify the placement of the question in \cite[Remark 2.11]{BSS} within the surrounding discussion.} for the abelian variety analog of Theorem~\ref{PnVanishing} in characteristic $0$; see Remark~\ref{rmkabvar} for more context. In fact, by the Lefschetz principle, it suffices to prove the analog for $k=\mathbf{C}$; in this case, we prove the analog more generally for compact complex tori.

\begin{theorem}
\label{AbVarVanChar0}
 Let $A$ be a compact complex torus over $\mathbf{C}$ of dimension $g$. For any constructible sheaf $F$ of $\Lambda$-modules on $A$, we have
\[ \colim_{n} H^i(A, [p^n]^* F) = 0\]
for $i > \dim(\mathrm{Supp}(F))$.
\end{theorem}

The proof relies on the Fourier--Mellin transform from \cite[\S 2]{BSS}. The latter organizes the cohomology of all rank $1$ twists of $F$ into a single coherent complex over the group algebra $\Lambda[\pi_1(A)]$. The key observation below is simply that the completion of $\Lambda[\pi_1(A)]$ at the origin is the inverse limit of the group algebras $\Lambda[\pi_1(A)/p^n]$; this permits translation of the commutative algebra properties of the Fourier--Mellin transform from \cite{BSS} into the vanishing statement required in the theorem using some standard results on Grothendieck duality in commutative algebra, recalled in Lemma~\ref{GrothDualCompat}.

\begin{proof}
Without loss of generality, we may assume $\Lambda=\mathbf{F}_p$. Moreover, by shifting, it suffices to show that if $G \in D^b_{cons}(A,\Lambda)$ is semiperverse, then $\colim_{n} R\Gamma(A, [p^n]^* G) $ is connective. By filtering $G$ for the perverse $t$-structure, we may also assume $G$ is perverse. In this case, we shall use the results from \cite[\S 2]{BSS} on the Fourier--Mellin transform to obtain the theorem.

 Let $S$ be the completion of the group algebra $R=\mathbf{F}_p[\pi_1(A)]$ at the origin, so $S$ is a formal power series in $2g$ variables. Note that each $R_n := \mathbf{F}_p[\pi_1(A)/p^n]$ is an artinian local ring and $S=\lim_n R_n$ via the natural map. More precisely, since the Frobenius on $R$ is induced by the multiplication by $p$ map on $\pi_1(A)$, the pro-systems $\{R_n\}$ and $\{R/\mathfrak{m}^{[p^n]}\}$ are naturally identified, where $\mathfrak{m}$ is the augmentation ideal $\mathfrak{m} = \ker(R \xrightarrow{\pi_1(A) \mapsto 1} \mathbf{F}_p)$, and $(-)^{[p^n]}$ denotes the $n$-th Frobenius power.

Write $K = \mathrm{FM}(\mathbf{D}G)^{\wedge} = \mathrm{FM}(\mathbf{D}G) \otimes_R S \in D^b_{coh}(S)$ for the completed stalk at $0$ of the Fourier-Mellin transform of the Verdier dual $\mathbf{D}G$ of $G$. Then the projective system $\{K_n := K \otimes_S R_n\}$ is identified with $\{R\Gamma(A, [p^n]^* \mathbf{D}G)\}$, where the transition maps are the trace maps. Taking $k$-linear duals termwise and using Verdier duality on the cohomological side, the inductive system $\{K_n^\vee\}$ is identified with the inductive system $\{R\Gamma(A, [p^n]^* G)\}$, where the transition maps are the natural pullback maps. Thus, our goal is to show that 
\[ \colim_n K_n^\vee\]
is connective. By standard arguments over noetherian local rings (see Lemma~\ref{GrothDualCompat} (3)), this colimit is identified with the local cohomology $R\Gamma_0(\mathbf{D}K)$ of the Grothendieck dual $\mathbf{D}K$ of $K$. As the local cohomological dimension of a finitely generated $S$-module is bounded above by its dimension, it suffices to show that $\dim \mathrm{Supp} (H^i(\mathbf{D}K)) \leq -i$. By \cite[Lemma 2.8]{BSS}, this is equivalent to showing that $K = \mathbf{D}^2K$ lies in $D^{\geq 0}$. But this was proven in \cite[Proposition 2.7]{BSS}, so we win.
\end{proof}

The following (well-known) result gives the compatibility of linear and Grothendieck duality over complete noetherian local $k$-algebras and was used above. The result is essentially a consequence of the complete-torsion equivalence as well as the fact that Grothendieck and linear duality coincide for finite length modules, though we give a direct argument avoiding this equivalence.

\begin{lemma}[Grothendieck and linear duality]
\label{GrothDualCompat}
Let $(S,\mathfrak{m})$ be a complete noetherian local $k$-algebra with residue field $k$ and normalized dualizing complex $\omega_S^\bullet$. Write $\mathbf{D}(-) = \mathrm{RHom}_S(-,\omega_S^\bullet)$ for $(-)^\vee = \mathrm{RHom}_k(-,k)$ for the displayed endofunctors of $D(S)$. 
\begin{enumerate}
\item For $M \in D(S)$, we have a natural isomorphism
\[ \mathbf{D}M \simeq R\Gamma_{\mathfrak{m}}(M)^\vee.\]
\item For $N \in D(S)$ and any finite length $S$-module $A$, we have a natural isomorphism
  \[\mathrm{RHom}_S(A,\mathrm{RHom}_k(N,k)) = (N \otimes_S A)^\vee.\]
\item For $M \in D(S)$, we have a natural isomorphism
\[ R\Gamma_{\mathfrak{m}}(\mathbf{D}M) \simeq \colim_n (M \otimes_S^L S/\mathfrak{m}^n)^\vee.\]
\end{enumerate}
\end{lemma}
\begin{proof}
For $N \in D(S)$, write $N[\mathfrak{m}^n] := \mathrm{RHom}_S(S/\mathfrak{m}^n,N)$, so $R\Gamma_{\mathfrak{m}}(N) = \colim_n N[\mathfrak{m}^n]$.
\begin{enumerate}
\item This follows from the sequence of identifications 
\begin{align*}
 \mathbf{D}M &= \mathrm{RHom}_S(M,\omega_S^\bullet) \\
 &= \mathrm{RHom}_S(\colim_n M[\mathfrak{m}^n], \omega_S^\bullet) \\
 &= \lim_n \mathrm{RHom}_S(M[\mathfrak{m}^n],\omega_S^\bullet) \\
 &= \lim_n \mathrm{RHom}_k(M[\mathfrak{m}^n],k) \\
 &= R\Gamma_{\mathfrak{m}}(M)^\vee,
 \end{align*}
 where the first line is definitional, the second follows as $R\Gamma_{\mathfrak{m}}(M) \to M$ induces an isomorphism after applying $\mathrm{RHom}_S(-,N)$ into any derived $\mathfrak{m}$-complete $N \in D(S)$ (such as the coherent complex $\omega_S^\bullet$), the third and fifth are formal, while the fourth is a consequence of Grothendieck duality for $\mathfrak{m}^n$-torsion $S$-modules agreeing with $k$-linear duality on underlying $k$-vector spaces.

\item Write $a^\times:D(k) \to D(S)$ for the functor $\mathrm{RHom}_k(S,-)$, so $a^\times$ is the right adjoint to restriction of scalars along $k \to S$. Then $\mathrm{RHom}_k(N,k) = \mathrm{RHom}_S(N,a^\times k)$, so 
 \begin{align*}
 \mathrm{RHom}_S(A,\mathrm{RHom}_k(N,k)) &= \mathrm{RHom}_S(A, \mathrm{RHom}_S(N, a^\times k)) \\
 &= \mathrm{RHom}_S(N \otimes_S A, a^\times k) \\
 &= (N \otimes_S A)^\vee,
 \end{align*}
  as wanted.

\item We can apply $R\Gamma_{\mathfrak{m}}(-)$ to the isomorphism in (1) to obtain
\begin{align*}
 R\Gamma_{\mathfrak{m}}(\mathbf{D}M) &= R\Gamma_{\mathfrak{m}}(\mathrm{RHom}_k(R\Gamma_{\mathfrak{m}}(M),k)) \\
 &= \colim_n \mathrm{RHom}_S(S/\mathfrak{m}^n, \mathrm{RHom}_k(R\Gamma_{\mathfrak{m}}(M),k)) \\
 &= \colim_n (R\Gamma_{\mathfrak{m}}(M) \otimes_S S/\mathfrak{m}^n)^\vee \\
 &= \colim_n (M \otimes_S^L S/\mathfrak{m}^n)^\vee,
 \end{align*}
 where the first comes from (1), the second from the standard calculation of local cohomology in terms of $\mathrm{Ext}$-groups over noetherian rings, the third from (2) applied to $N=R\Gamma_{\mathfrak{m}}(M)$ and $A=S/\mathfrak{m}^n$ for varying $n$,  and the last is a consequence of $R\Gamma_{\mathfrak{m}}(M) \to M$ inducing an isomorphism after applying $- \otimes_S S/\mathfrak{m}^n$. \qedhere
 \end{enumerate}
\end{proof}
\bibliographystyle{amsalpha}
\bibliography{mybib}

\end{document}